\documentclass[letterpaper,11pt]{amsart}
\usepackage{amsmath,amssymb,amsthm,pinlabel,tikz}
\usepackage{verbatim,mathtools}
\usepackage{tikz-cd}
\newcommand{\nc}{\newcommand}
\nc{\dmo}{\DeclareMathOperator}
\dmo{\ra}{\rightarrow}
\dmo{\N}{\mathbb{N}}
\dmo{\Z}{\mathbb{Z}}
\dmo{\R}{\mathbb{R}}
\dmo{\C}{\mathcal{C}}
\dmo{\AC}{\mathcal{AC}}
\dmo{\Mod}{Mod}
\dmo{\PMod}{PMod}
\dmo{\B}{B}
\dmo{\PB}{PB}
\dmo{\Sp}{Sp}
\dmo{\I}{\mathcal{I}}
\dmo{\el}{\ell_{\C}}
\dmo{\NN}{\mathcal{N}}
\dmo{\Tr}{Tr} 
\dmo{\rk}{rk}
\dmo{\Aut}{Aut}

\dmo{\Teich}{Teich}
\dmo{\Ind}{Ind}
\dmo{\cd}{cd}

\dmo{\forget}{Forget}
\dmo{\Homeo}{Homeo}
\dmo{\Diffeo}{Diffeo}
\dmo{\push}{Push}
\dmo{\capd}{Cap}
\dmo{\CG}{CG}
\dmo{\UCG}{UCG}
\dmo{\FCG}{FCG}
\dmo{\PFB}{PFB}
\dmo{\BG}{B}
\dmo{\PBG}{PB}

\nc{\nt}{\newtheorem}

\newtheorem{thm}{{\bf Theorem}}[section]
\newtheorem{lem}[thm]{{\bf Lemma}}

\newtheorem{prop}[thm]{{\bf Proposition}}

\newtheorem{remark}[thm]{Remark}

\newtheorem{conj}[thm]{Conjecture}
\newtheorem{definition}[thm]{Definition}
\numberwithin{equation}{section}

\title[Torelli groups of partitioned surfaces]{Topological and dynamical properties of Torelli groups of partitioned surfaces}

\author{Hyungryul Baik}
\address{%
		Department of Mathematical Sciences, KAIST,
		291 Daehak-ro Yuseong-gu, Daejeon, 34141, South Korea 
}
\email{%
        hrbaik@kaist.ac.kr
}

\author{Hyunshik Shin}

\address{%
        Department of Mathematics, University of Georgia,
		Athens, GA 30602
}
\email{%
        hyunshik.shin@uga.edu
}

\author{Philippe Tranchida}
\address{%
		Department of Mathematical Sciences, KAIST,
		291 Daehak-ro Yuseong-gu, Daejeon, 34141, South Korea 
}
\email{%
        ptranchida@kaist.ac.kr
}

\begin{document}
\begin{abstract}
Putman introduced a notion of a partitioned surface which is a surface with boundary with 
decoration restricting how the surface can be embedded into larger surfaces, and defined the Torelli group of the partitioned surfaces. In this
paper, we study some topological and dynamical aspects of the Torelli groups of partitioned surfaces. More precisely, first we obtain upper and lower bounds on the cohomological dimension of Torelli groups of partitioned surfaces and show that those two bounds coincide when at most three boundary components are grouped together in the partition of the boundary. Second, we study the asymptotic translation lengths of Torelli groups of partitioned surfaces on the corresponding curve complexes. We show that the minimal asymptotic translation length asymptotically behaves almost like the reciprocal of the Euler characteristic of the surface. This generalizes the previous result of the first and second authors on Torelli groups for closed surfaces. 

\end{abstract}

\maketitle

%
%

\medskip
\section{Introduction}	\label{section:introduction}

Let $S$ be a connected orientable surface of finite type. When it
is closed with genus $g$, we denote it by $S$. In this
paper, we will also consider surfaces with punctures or boundary
components but not both at the same time. For notation, we use
$S_{g}^n$ to denote the surface of genus $g$ with $n$ boundary
components, and $S_{g,n}$ to denote the surface of genus $g$
with $n$ punctures. In any case, we always assume $\chi(S) < 0$. 
The \textit{mapping class group} of $S$, denoted by $\Mod(S)$, is the 
group of homotopy classes of orientation-preserving homeomorphisms of $S$.
For a surface $S_g^n$ with boundary components, $\Mod(S_g^n)$ consists of 
orientation-preserving homeomorphisms up to homotopy fixing boundary pointwise. 

For a closed surface $S_g$, one of the important proper normal subgroups of $\Mod(S)$ is the \textit{Torelli group} $\I(S_g)$,
which is the kernel of the following symplectic representation
$$\Psi: \Mod(S_g) \ra \Sp(2g,\Z).$$
This homomorphism is induced by the action on the first homology $H_1(S_g,\Z)$
and the algebraic intersection number gives a symplectic structure.
For more discussion on Torelli group, see \cite{FarbMargalit12} or \cite{Johnson83}.

For a surface $S_g^n$ with boundary components, if $n>1$, the algebraic intersection number becomes degenerate.
Hence there is no such symplectic representation and the analogous definition of $\I(S_g^n)$ does not work.
In \cite{putman2007cutting}, Putman introduced the notion of a partitioned surface to 
develop a ``Torelli functor'' on the category of partitioned surfaces and extend the usual definition of Torelli group
of a closed surface.
For a surface $S_g^n$ with a partition $P$ of boundary components, let $\I(S_g^n,P)$ be the Torelli group  for the pair $(S_g^n,P)$ in the sense of Putman. When $\partial S = \{b_1, \ldots, b_n\}$, we set notations for two extreme cases for $P$ as $P_{\max}=\{\{b_1\}, \{b_2\}, \ldots, \{b_n\}\}$ and $P_{\min}=\{ \{ b_1, b_2, \ldots, b_n \} \}.$ For the definition and review of basic facts, see Section \ref{sec:putman} or \cite{putman2007cutting}. 

For any partition $P$, we compute bounds for the cohomological dimension of $\I(S_g^n,P)$. In [BBM], the authors showed that the cohomological dimension of the Torelli group of a closed surface is equal to $3g -5$. The fact that $3g -5$ is a lower bound for the cohomological dimension was already shown long ago by Mess in an unpublished paper \cite{Mess90}. He used the theory of Poincare duality groups for constructing, by induction, subgroup of the wanted dimension. The authors of [BBM] then showed that $3g -5$ is an upper bound by using the action of the Torelli group on the complex of reduced cycles that they defined in the same paper.

We will deduce a lower bound from the lower bound for the closed surface case and a generalized Birman exact sequence for Torelli groups of partitioned surfaces. For the upper bound, we will extend the Poincare duality groups of Mess by some free abelian subgroup for bridging the gap between the closed case and the boundary case.

Let $S_g^n$ be a surface of genus $g$ with $n$ boundary components. Let $P$ be a partition of its boundaries. We define $P^1 := \{ p \in P \mid |p| = 1 \}$ and $P^2 := \{ p \in P \mid |p| = 2 \}$. Our main theorem about cohomological dimension is 

\begin{thm}
\label{main}
Let $n \geq 1$, and let $|P|$ denote the cardinality of $P$. Then $$3g -4 + 3|P|-|P^1| \leq \cd(\I_g^n,P) \leq 3g -4 +n +(|P^1| + |P^2|)$$
\end{thm}

We remark that this lower bound and upper bounds coincide if the partition $P$ contains only elements of cardinality at most $3$, giving a precise formula for those cases.

Next, we discuss the least asymptotic translation length of 
proper normal subgroups of $\Mod(S_g^n)$ on the curve graph.
This is motivated by the action of Torelli group of a closed surface $S_g$ on the curve graph $\C(S_g)$.
Let $\C(S_g)$ be the \textit{curve graph} of $S_g$ with path metric $d_{\C}(\cdot,\cdot)$,
i.e., each edge in $\C(S_g)$ has length 1.
Let $\ell_{\C}(f)$ be the \textit{asymptotic translation length} of $f \in \Mod(S_g)$ defined by
$$ \ell_{\C} (f) = \liminf_{j \ra \infty} \frac{d_{\C}(\alpha, f^j(\alpha))}{j},$$
where $\alpha$ is an element in $\C(S)$. Note that $\ell_{C}(f)$ is independent of
the choice of $\alpha$.
For any $H \subset \Mod(S_g)$, define $$L_{\C}(H) = \min \{ \el(f) : \, f \in H, \ \textrm{pseudo-Anosov} \}.$$
In \cite{BaikShin18}, the first and second authors proved that 
$$L_{\C}(\I(S_g)) \asymp \frac{1}{g},$$
that is, there are constants $C_1$ and $C_2$, independent of $g$, such that 
$$ \frac{C_1}{g} \leq L_{\C}(\I(S_g)) \leq \frac{C_2}{g}.$$
On the contrary, Gradre--Tsai \cite{GadreTsai11} showed that $L_{\C}(\Mod(S_g)) \asymp 1/g^2$.
Hence the minimal asymptotic translation lengths from $\Mod(S_g)$ and $\I(S_g)$ approach 0
at a different rate, $1/g^2$ and $1/g$, respectively.

This is analogous to the previous results of \cite{Penner91} and \cite{FarbLeiningerMargalit08} regarding 
the stretch factor $\lambda$ of pseudo-Anosov mapping classes.
For any $H \subset \Mod(S_g)$, let us define
$$L(H) = \min \{ \log(\lambda(f)) : \, f \in H \ \textrm{is pseudo-Anosov} \}.$$
Notice that $L(\Mod(S_g))$ can be thought of as the length spectrum of the moduli space of genus $g$ Riemann surfaces.
Penner \cite{Penner91} proved that $L(\Mod(S_g)) \asymp 1/g$. On the contrary, Farb--Leininger--Margalit \cite{FarbLeiningerMargalit08} showed that $L(\I(S_g)) \asymp 1$, and that for level $m$ congruence subgroup $\Mod(S_g)[m]$,
$L(\Mod(S_g)[m]) \asymp 1$ with $m \geq 3$.
Later, Lanier--Margalit \cite{LanierMargalit18} generalized this result and showed that for any proper normal subgroup $N$ of $\Mod(S_g)$,
$L(H) > \log(\sqrt{2})$ for $g >3$.

Inspired by previous works, we generalize the result of \cite{BaikShin18} about Torelli groups to the case of surfaces with boundary components. Our main result along this line is 
\begin{thm} \label{thm:mainthmlength} 
We have 
$$\frac{1}{|\chi(S_g^n)|} \lesssim L_{\C} (\I(S_g^n,P)) \lesssim \frac{1}{|\chi(S_g^n)| -(n-|P|)}.$$ 

In particular, $L_{\C} (\I(S_g^n,P)) \asymp \frac{1}{|\chi(S_g^n)|}$. 
\end{thm}

Note that Gradre--Tsai proved that 
$$L_{\C}(\Mod(S_g^n)) > \frac{1}{18(2g-2+n)^2+30(2g-2+n)-10n}.$$

Hence it is natural to ask if 
$L_{\C}(\Mod(S_g^n)) \asymp \frac{1}{|\chi(S_g^n)|^2}.$ We show that it is not the case. 
\begin{thm} \label{thm:stopdreaming}
Suppose $g < (1/4 - \epsilon) n$ for arbitrarily small $\epsilon > 0$. Then 
$L_{\C}(\Mod(S_g^n)) \asymp 1/|\chi(S_{g,n})|$.
\end{thm} 
It would be interesting to know if we have $L_{\C}(\Mod(S_g^n)) \asymp 1/|\chi(S_{g,n})|$ without any restriction on $(g, n)$.

In Section \ref{sec:putman}, we review the notion of partitioned surfaces and Putman's construction of Torelli groups for partitioned surfaces. 
In section \ref{sec:cohomid}, we obtain bounds on the cohomological dimensions on Torelli groups for partitioned surfaces in the sense of Putman. 
In Section \ref{sec:lowerboundlength}, we obtain the lower bound for Theorem \ref{thm:mainthmlength}. In Section \ref{sec:upperboundlength}, we complete the proof of Theorem \ref{thm:mainthmlength} by obtaining the upper bound. 
In Section \ref{sec:future}, we discuss the current state of the art and prove Theorem \ref{thm:stopdreaming}. 

\medskip
\subsection*{Acknowledgements}
We thank Inhyeok Choi, Chenxi Wu for helpful discussions. 
The first author was supported by the National Research Foundation of Korea(NRF) grant funded by the Korea government(MSIT) (No. 2020R1C1C1A01006912).

%
%

\medskip
\section{Preliminaries}\label{sec:putman} 

\subsection{Cohomological dimension of groups}

We first recall the definition of cohomological dimension, even though we will not actually make any use of the concrete definition.
\begin{definition}
Let $G$ be a group. The \textit{cohomological dimension} of $G$, denoted $\cd(G)$, is the smallest integer $n$ such that
\begin{enumerate}
    \item for any $G$-module $M$ and for every $i > n$, we have $H^i(G,M) = 0 $
    \item There exists some $G$-module $M$ with $H^n(G,M) \neq 0$
\end{enumerate}
\end{definition}

More importantly, we will repeatedly use two key properties of cohomological dimensions that we list here.

\begin{prop}[Monotonicity]
Let $H \subset G$ be a subgroup. Then $\cd(H) \leq \cd(G)$.
\end{prop}

\begin{prop}[Subadditivity]
Let 
\begin{tikzcd}
1 \arrow[r] & H \arrow[r] & G \arrow[r] & Q \arrow[r] & 1
\end{tikzcd}
be a short exact sequence of groups. Then, we have that $\cd(G) \leq \cd(H) + \cd(Q)$.
\end{prop}

We will also need to use some theory about Poincare duality groups. For defining Poincare duality groups, we first need one definition about a finiteness property of groups.

\begin{definition}
A group $G$ is of \textit{type FP} if there exists a resolution of $\mathbb{Z}$ of finite length by finitely generated $\mathbb{Z}G$-modules:
$$
0 \to P_n \to \cdots \to P_0 \to \mathbb{Z} \to 0
$$
such that each $P_i$ is projective.
\end{definition}

A Poincare duality group is a group $G$ of type $FP$ whose cohomology groups satisfy some additional structure. More precisely:
\begin{definition}
A group $G$ is a \textit{Poincare duality group of dimension n} if 
\begin{enumerate}
    \item $G$ is of type FP
    \item $H^i(G,\Z G) = 0 $ if $i \neq n$ and $H^n(G,\Z G) = \Z$
\end{enumerate}
\end{definition}

They key feature of Poincare duality groups we will make use of is that they improve the subadditivity of cohomological dimension.

\begin{prop}
Let 
\begin{tikzcd}
1 \arrow[r] & H \arrow[r] & G \arrow[r] & Q \arrow[r] & 1
\end{tikzcd}
be a short exact sequence of groups such that $H$ and $Q$ are Poincare duality groups. Then $G$ is also a Poincare duality group and $\cd(G) = \cd(H) + \cd(Q)$.
\label{Poincare}
\end{prop}
Finally, we need to be able to recognise some Poincare duality group when we encounter them.

\begin{prop}
The fundamental group of a closed, aspherical n-manifold and free abelian groups of rank $n$ are both Poincare duality group of dimension n.
\label{manifold}
\end{prop}

\subsection{Putman's Construction of Torelli groups for surfaces with boundary} 

Let $S_g^n$ be a surface of genus $g$ with $n$ boundary components, say with the labels $\{b_1, \ldots, b_n\}$. 
Let $P$ be a partition of the set of boundary components of $S_g^n$, then we call the pair $(S_g^n, P)$ a partitioned surface. 
For instance, one can think of an example where $n = 7$ and $P = \{ \{b_1, b_2, b_3\}, \{b_4, b_5\}, \{b_6\}, \{b_7\} \}$. (See Figure \ref{fig:capping}.)

Let $(S_g^n, P)$ be a partitioned surface. As in \cite{putman2007cutting}, we define a capping of $(S_g^n, P)$ as an embedding $S_g^n \xhookrightarrow{} S_g$ 
where the set of the sets of boundary components of the connected components of $S_g \setminus S_g^n$ is exactly the partition $P$.
The Torelli group $\mathcal{I}(S_g^n, P)$ of the partitioned surface $\mathcal{I}(S_g^n, P)$ is defined by
$$ \mathcal{I}(S_g^n, P) := \iota_*^{-1} (\mathcal{I}_g) < \Mod(S_g^n).$$
for any capping $\iota: S_g^n \xhookrightarrow{} S_g$.
Putman proved that this definition is independent of the choice of capping $\iota$.

\begin{figure}[t]	
\centering
	\begin{tikzpicture}[scale=.6]
		\draw[thick] (1,5.4) to [out=-10, in=-170] (7,5.3);
        \draw[thick] (1,0) to [out=10, in=170] (7,0);
        \foreach \y in {0, 2, 4}{
        \draw[thick] (1,0+\y) to [out=70, in=-70] (1,1.4+\y);
        \draw[thick] (1,0+\y) to [out=110, in=-110] (1,1.4+\y);
        }
        \foreach \y in {0, 1.4, 2.8, 4.2}{
        \draw[thick] (7,0+\y) to [out=70, in=-70] (7,1.1+\y);
        \draw[thick] (7,0+\y) to [out=110, in=-110] (7,1.1+\y);
        }
        \foreach \y in {0,2}{
        \draw[thick] (1,1.4+\y) to [out=10, in=-10] (1,2+\y);
        }
        \foreach \y in {0,1.4,2.8}{
        \draw[thick] (7,1.1+\y) to [out=180, in=-180] (7,1.4+\y);
        }
        \draw[->,thick] (8.2,2.6) to (10.2,2.6);
        \draw[thick] (10+3,5.4) to [out=-10, in=-170] (16+3,5.3);
        \draw[thick] (10+3,0) to [out=10, in=170] (16+3,0);
        \foreach \y in {0, 2, 4}{
        \draw[thick] (10+3,0+\y) to [out=70, in=-70] (10+3,1.4+\y);
        \draw[thick, dotted] (10+3,0+\y) to [out=110, in=-110] (10+3,1.4+\y);
        }
        \foreach \y in {0, 1.4, 2.8, 4.2}{
        \draw[thick, dotted] (16+3,0+\y) to [out=70, in=-70] (16+3,1.1+\y);
        \draw[thick] (16+3,0+\y) to [out=110, in=-110] (16+3,1.1+\y);
        }
        \foreach \y in {0,2}{
        \draw[thick] (10+3,1.4+\y) to [out=0, in=0] (10+3,2+\y);
        \draw[thick] (10+3,1.4+\y) to [out=180, in=180] (10+3,2+\y);
        }
        \foreach \y in {0,1.4,2.8}{
        \draw[thick] (16+3,1.1+\y) to [out=180, in=-180] (16+3,1.4+\y);
        }
        \draw[thick] (16+3,1.1+2.8) to [out=0, in=0] (16+3,1.4+2.8);
        \draw[thick] (13,0) to [out=180, in=-180] (13,5.4);
        \draw[thick] (19,2.8) to [out=0, in=0] (19,5.3);
        \draw[thick] (19,1.4) to [out=0, in=0] (19,2.5);
        \draw[thick] (19,0) to [out=0, in=0] (19,1.1);
        \foreach \y in {0,12}{
        \node at (1.5+\y, .7) {$b_1$};
        \node at (1.5+\y, 2.7) {$b_2$};
        \node at (1.5+\y, 4.6) {$b_3$};
        \node at (6.5+\y, .7) {$b_7$};
        \node at (6.5+\y, 1.9) {$b_6$};
        \node at (6.5+\y, 3.3) {$b_5$};
        \node at (6.5+\y, 4.6) {$b_4$};
        }
        \foreach \x in {0, 12} {
        \foreach \y in {0, 2} {
        \draw[thick] (3.5+\x,1.5+\y) to [out=-30,in=-150] (4.5+\x,1.5+\y);
        \draw[thick] (3.6+\x,1.45+\y) to [out=30,in=150] (4.4+\x,1.45+\y);
        }
        }
        \draw[thick] (3.5+8.2,1.5+1) to [out=-30,in=-150] (4.5+8.2,1.5+1);
        \draw[thick] (3.6+8.2,1.45+1) to [out=30,in=150] (4.4+8.2,1.45+1);
	\end{tikzpicture}
	\caption{A capping of $S_2^7$ with $P = \{ \{b_1, b_2, b_3\}, \{b_4, b_5\}, \{b_6\}, \{b_7\} \}.$}
	\label{fig:capping}
\end{figure}
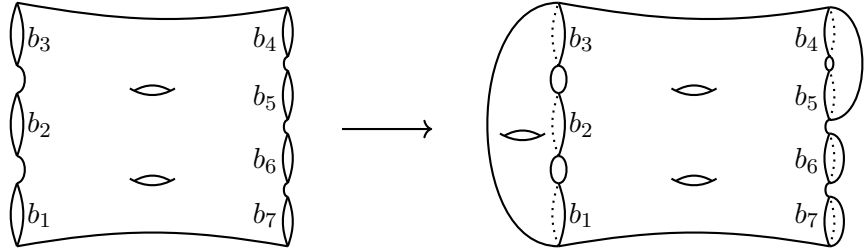


\subsection{Action on Homology groups}
Putman defined some version of homology group $H_1^P(S_g^n)$ depending only on the surface and the partition $P$. 
For completion, we describe the definition of $H_1^P(S_g^n)$ here (see \cite{putman2007cutting} for detail). 

Consider a partitioned surface $(S,P)$ and enumerate the partition $P$ as $$P=\{ \{ b_1^1, \ldots, b_{k_1}^1\}, \cdots, \{b_1^m, \ldots, b_{k_m}^m \} \}.$$
Orient the boundary components $b_i^j$ so that $\sum_{i,j}[b_i^j] = 0$ in $H_1(S)$. Define
\begin{align*}
	\partial H_1^P( S) &= \langle [b_1^1]+\cdots +[b_{k_1}^1], \cdots, [b_1^m]+ \cdots + [b_{k_m}^m] \rangle \subset H_1(S)
\end{align*}
Let $Q$ be a set containing one point from each boundary component of $S$. Define $H_1^P(S)$ to be equal to the image of the following 
submodule of $H_1(S,Q)$ in $H_1(S,Q)/\partial H_1^P(S)$:
\begin{align*}
	\langle \{ [h] \in H_1(S,Q) | \ &\textrm{either $h$ is a simple closed curve or $h$ is a properly}\\
	& \textrm{embedded arc from $q_1$ to $q_2$ with $q_1, q_2 \in Q$ lying}\\
	& \textrm{in boundary components $b_1$ and $b_2$ with $\{b_1, b_2\} \subset P$}\\
	& \textrm{for some $p\in P$} \} \rangle.
\end{align*}

Putman proved that an element of $\Mod(S_g^n)$ belongs to $\mathcal{I}(S_g^n, P)$
if and only if it acts trivially on $H_1^P(S_g^n)$. In particular, this shows that $\mathcal{I}(S_g^n, P)$ is well-defined, not depending on the choice of the capping $\iota$.

%
%
\medskip
\section{Estimates on the cohomological dimension} \label{sec:cohomid}

We begin by a brief discussion on the arguments used in \cite{BestvinaBuxMargalit10} for computing the cohomological dimensions of a surface with only one boundary, as the ideas used for the general case are somewhat similar. The authors in \cite{BestvinaBuxMargalit10} first prove that $\cd(\I(S_g)) = 3g -5$.
Then, the Birman exact sequence

\begin{center}
\begin{tikzcd}
1 \arrow[r] & \pi_1(US_g) \arrow[r] & \I(S_g^1) \arrow[r] & \I(S_g) \arrow[r] & 1
\end{tikzcd}
\end{center}

immediately yields that $\cd(\I(S_g^1)) \leq 3(g+1) -5$. Indeed, $US_g$ is a closed 3-manifold, and thus has cohomological dimension $3$. We thus get the desired inequality by subadditivity of cohomological dimensions. 

The inequality in the other direction uses Poincare duality groups and was first proven in \cite{Mess90}. It was also shown in \cite{BestvinaBuxMargalit10} that $\I(S_g)$ contains a Poincare duality group $\Gamma$ of cohomological dimension $3g -5$. Restricting the Birman exact sequence to $\Gamma$, we get

\begin{center}
\begin{tikzcd}
1 \arrow[r] & \pi_1(US_g) \arrow[r] & \Gamma' \arrow[r] & \Gamma \arrow[r] & 1
\end{tikzcd}
\end{center}
where $\Gamma'$ is the preimage of $\Gamma$ under the capping homomorphism. By proposition \ref{manifold}, $\pi_1(US_g)$ also is a Poincare duality group. Hence, using proposition \ref{Poincare}, we get that $\cd(\Gamma') = 3g - 5 +3 = 3(g+1) -5$. Since $\Gamma'$ is a subgroup of $\I(S_g^1)$, we deduce that $\cd(\I(S_g^1)) \geq 3(g+1)-5$. Therefore, this proves that $\cd(\I(S_g^1)) = 3(g+1) -5$.

Since here $P = P^1$ and $ = |P| = |P^1| = n = 1$, both the lower and upper bound of theorem $\ref{main}$ are equal to $3g -4 +2 = 3(g+1) -5$, concluding the case $n=1$. 
\subsection{The upper bound}

Let $S_g^n$ be a surface of genus $g$ with $n$ boundary components and consider $S_g$ to be obtained from $S_g^n$ by gluing a disk to each boundary. We then have a map $\capd_n \colon \Mod(S_g^n) \to \Mod(S_g)$ obtained by extending mapping classes in $\Mod(S_g^n)$ by the identity on each of the glued disks. 

We would like to understand the kernel of $\capd_n$ as well as the kernel of the restriction of $\capd_n$ to $\I(S_g^n,P)$. The case of the map $\capd \colon \Mod(S_g^n) \to \Mod(S_g^{n-1})$ is already well understood. Indeed, the Birman exact sequence shows that the kernel of $\capd$ is isomorphic to the fundamental group of the unitary tangent bundle of $S_g^{n-1}$ and the kernel of the restriction of $\capd$ to $\I(S_g^n,P)$ was investigated in \cite{putman2007cutting}. We carry a similar investigation for $\capd_n$ here.

To understand the kernel of $\capd_n$, we need to introduce the notions of configurations spaces and framed configuration spaces and their fundamental groups.

The $n$ dimensional configuration space of a space $X$, denoted $\CG_n(X)$, is the set $\{(p_1,\cdots,p_n) \in X^n \mid p_i \neq p_j \forall i \neq j \}$. In other words, $\CG_n(X)$ is the set of all $n$-tuples of distinct points in $X^n$. The pure braid group on $n$ strands over $X$, denoted $PB_n(X)$ is defined to be the fundamental group of $\CG_n(X)$.

We define the framed configuration space of a space $X$, denoted by $\FCG(X)$, to be the set $\{(p,v_1, \cdots ,v_n) \in CG(S_g) \times S_1^n\}$. We think of it as the space of embeddings of $n$ distinct points into $X$, with a unit tangent vector at each point. The pure framed braid group on $n$ strands over $X$, denoted by $\PFB_n(X)$, is the fundamental group of $\FCG(X)$.

Armed with these new concepts, we can describe the kernel of $\capd_n \colon$
\newline  $\Mod(S_g^n) \to \Mod(S_g)$.
We note that this is proven in \cite{BellingeriGervais12}, but we recall the proof here for completeness.

\begin{thm}
The kernel of $\capd_n \colon \Mod(S_g^n) \to \Mod(S_g)$ can be identified with $\PFB_n(S_g)$. Hence, the following sequence is exact
$$
1 \to \PFB_n(S_g) \to \Mod(S_g^n) \to \Mod(S_g) \to 1
$$
\end{thm}
\begin{proof}
We consider $S_g$ to be obtained from $S_g^n$ by gluing a disk on each boundary component and we let $p_1, \cdots, p_n$ be $n$ points, one in each disk we glued. We also choose a unit tangent vector $v_i$ at each $p_i$.
Let $\Diffeo^+(S_g)$ be the group of orientation preserving diffeomorphisms of $S_g$ and
define $\psi \colon \Diffeo^+(S_g) \to \FCG(S_g)$ by $\psi(f) = (f(p_1), \cdots, f(p_n), f_*(v_1) ,\cdots, f_*(v_n))$. Then $\psi$ is a fiber bundle with fiber the group of diffeomorphisms of $S_g$ fixing $\{p_1, \cdots, p_n\}$ pointwise and fixing a unit tangent vector at each of these points. Thus, these diffeomorphisms fix the regular neighborhoods of each of the $p_i$'s. We can thus identify the fiber with $\Diffeo^+(S_g^n)$. We then get the desired result by looking at the long exact sequence of homotopy groups associated to this fiber bundle.
\end{proof}

\begin{remark}
It is shown in \cite{BellingeriGervais12} that $\PFB_n(S_g)$ is a semidirect product of $\PB_n(S_g)$ with $\mathbb{Z}^n$.
\end{remark}

We now investigate what the kernel of the restriction of $\capd_n$ to $\I(S_g^n,P)$. Let us define $P^1 = \{p \in P \mid |p| = 1\}$ and $P^2 = \{p \in P \mid |p| = 2\}$.

\begin{thm}
The kernel $K$ of $\capd_n \colon \I(S_g^n,P) \to \I(S_g)$ fits in a short exact sequence
$$
1 \to \mathbb{Z}^m \to K \to K'\to 1
$$
where $K'$ is $\pi(K)$ and $m = |P^1| + |P^2|$. 
\end{thm}

\begin{proof}
We have a short exact sequence
$$
1 \to \mathbb{Z}^n \to \PFB_n(S_g) \to \PBG_n(S_g) \to 1
$$
and $K$ is a subgroup of $\PFB_n(S_g)$. Hence, we want to understand the rank of $\mathbb{Z}^n \cap K$. The generators of $\mathbb{Z}^n$ are represented, as elements of $\Mod(S_g^n)$, by the Dehn twists around the boundary components of $S_g^n$. We thus only need to check which compositions of these Dehn twists are in $\I(S_g^n,P)$. Clearly, for any $p \in P$ of cardinality $1$, the corresponding Dehn twist is in $\I(S_g^n,P)$ since it is a $P$-separating twist. Also, for any $ p = \{ b_i,b_j\}$, we have that $T_{b_i}T_{b_j}^{-1}$ also is in $\I(S_g^n,P)$ since it is a $P$-bounding pair. We now show that the only elements in $\mathbb{Z}^n \cap K$ are words in these $P$-separating twists and in these $P$-bounding pairs.

For proving that no other composition of such Dehn twist is in $\I(S_g^n,P)$, we will look at their action on some arcs, which are elements of $H_1^P(S_g^n)$. Suppose that $f$ is a product of powers of Dehn twists around the boundaries of $S_g^n$. If $T_{b_1}^{k_1}T_{b_2}^{k_2}$ appears in $f$ with $k_1 + k_2 \neq 0$ for some $\{b_1,b_2\} \in P$, then we choose any arc $a$ from $b_1$ to $b_2$ and note that $f_*([a]) = T_{b_1}^{k_1}T_{b_2}^{k_2}([a]) \neq [a]$.
Similarly, if $T_b^k$ appears in $f$ for some $b \in p \in P$ with $|p| \geq 3$, then choosing any $b' \neq b \in p$ and any arc $a$ from $b$ to $b'$, we get that $f_*([a]) = T_b^k T_{b'}^{k'}([a]) = [a] + k'[b'] +k[b] \neq [a]$, where the last inequality comes from the fact that $[b]$ and $[b']$ are independent in $H_1^P(S_g^n)$.

\end{proof}


We thus have a short exact sequence 

$$ 1 \to K \to \I(S_g^n,P) \to \I(S_g) \to 1 $$

with $K$ being an extension of a subgroup of $\PBG_n(S_g)$ by $\mathbb{Z}^m$. In \cite{GonccalvesGuaschiMaldonaldo18}, the authors show that $\cd(\PBG_n(S_g)) = n+1$. Hence, we obtain that $\cd(K) \leq n+1+m = n+1+(|P^1| + |P^2|)$. This proves that $\cd(\I(S_g^n,P)) \leq 3g -5 + n + 1 + (|P^1| + |P^2|) = 3g -4 + n +(|P^1| + |P^2|)$, as desired. 

\subsection{The lower bound}

We will prove the lower bound by induction on $n$. To be precise, the induction hypothesis is that, for all partitioned surfaces with $n-1$ boundaries, there exists a Poincare duality group $\Gamma \subset \I(S_g^{n-1},P')$ such that $\cd(\Gamma) = 3g -4 + 3|P'| -|P'^1|$. Recall that we showed this to hold for $n =1$.

Let $S_g^n$ be a surface of genus $g$ with $n$ boundaries and let $P$ be a partition of the boundaries. 

Suppose that there exist two boundaries $b$ and $b'$ that are both capped on their own (i.e: both $\{b\}$ and $\{b'\}$ are elements of $P$). Let us embed $S_g^{n-1}$ in $S_g^n$ as suggested by the left part of figure \ref{lower bound 1}. More precisely, the blue curve in figure \ref{lower bound 1} cut $S_g^{n}$ into to surfaces $S_1$ and $S_2 = S_0^3$. We then identify $S_g^{n-1}$ with $S_1$. Let us also number the boundaries of $S_g^n$ so that $b = \delta_1$ and $b' = \delta_2$. We consider the capping $P'$ of $S_g^{n-1}$ induced by $P$. This implies that $\I(S_g^{n-1},P')$ is a subgroup of $\I(S_g^n,P)$ (see \cite{putman2007cutting}). By induction, there exists a Poincare duality group $\Gamma$ of dimension $3g -4 + 3|P'| -|P'^1|$ in $\I(S_g^{n-1},P')$.

Since $|P'| = |P| -1$, we need to add 2 dimensions to achieve the wanted lower bound. Indeed $3|P'| = 3|P| -3$ and $|P^1| =|P'^1|+1$. This is done by considering the group $G = \langle \Gamma, T_{\delta_1}, T_{\delta_2} \rangle$. This group $G$ is clearly a subgroup of $\I(S_g^n,P)$ and it is easy to show that $ G \equiv \Z^2 \bigoplus \Gamma$. This can be done, for example, by considering the action of $G$ on the arc $\alpha$, drawn in figure \ref{lower bound 1}. Hence, since $\Z^2$ is a Poincare duality group, proposition \ref{Poincare} shows that $\cd(G) = \cd(\Gamma) + 2$ and hence, by monotonicity of cohomological dimensions, that $\cd(S_g^n,P) \geq \cd(G) = 3g -4 +3(|P|-1) - | P^1 |)$ as desired.

Suppose instead that there is a pair of boundaries $b,b'$ such that $\{b,b'\} \in P$. We embed $S_g^{n-1}$ into $S_g^n$ the same way as before (see the right part of figure \ref{lower bound 1})and let $P'$ be the induced capping. Once again, we order the boundaries of $S_g^n$ such that $b = \delta_1$ and $b' = \delta_2$. This time, we only need to add one dimension, since $3|P| = 3|P'|$ and $|P^1| =|P'^1| +1$. Therefore, it is sufficient to add the bounding pair $T_{\delta_1} T_{\delta_2} ^{-1}$ to $\Gamma$, the Poincare duality group of dimension $3g-4 +3|P'| -|P'^1|$ in $\I(S_g^{n-1},P')$. As before, the group $G = \langle \I(S_g^{n-1}), T_{\delta_1}T_{\delta_2}^{-1} \rangle$ is isomorphic to $\Z \oplus \I(S_g^{n-1})$, and thus has one more cohomological dimension than $\I(S_g^{n-1})$, concluding the proof.


\begin{figure}
    \centering
    \includegraphics[scale = 0.4]{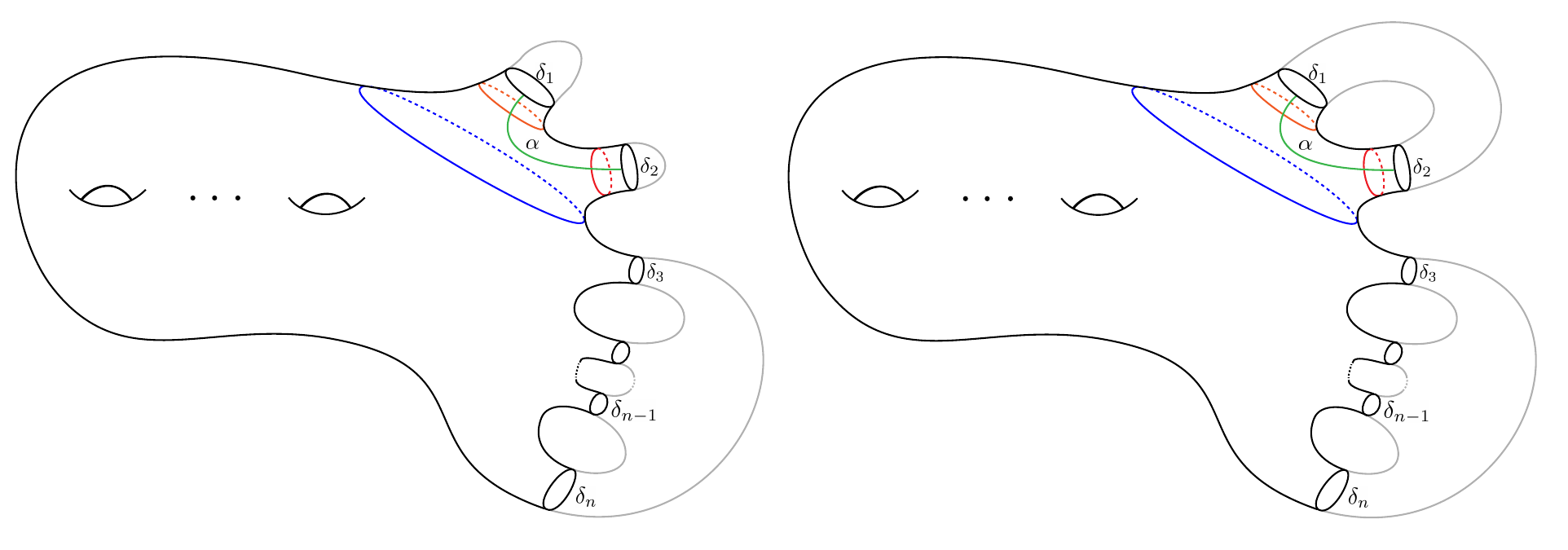}
    \caption{$S_g^n$ is the surface delimited by the black line and the cappings are suggested by the shaded lines. The blue curve indicates how we embed $S_g^{n-1}$ into $S_g^n. $}
    \label{lower bound 1}
\end{figure}

Putting those two cases together, we can show that theorem \ref{main} holds whenever each element of the partition $P$ has cardinality $1$ or $2$.

In other words, we proved that each $p \in P$ adds exactly $2$ dimensions if $p = \{b\}$ and that it adds exactly $3$ dimensions if $|p| =2$. It remains to show that $p$ adds at least $3$ dimensions if $p$ contains $2$ or more elements. 

The proof is almost identical to the case of $|p| = 2$. We again embed $S_g^{n-1}$ into $S_g^n$ as suggested by figure \ref{lower lower bound}. Let $\gamma$ be a simple closed curve cutting of all boundaries in $p$ except for $\delta_1$ (see figure \ref{lower lower bound}). We then consider the group $G = \langle \Gamma , T_{\delta_1} T_\gamma ^{-1}\rangle$ where $\Gamma$ is the Poincare duality group of the right dimension in $\I(S_g^{n-1},P')$. Once again, $G \equiv \Z \oplus \I(S_g^{n-1})$ and this shows that $\cd(\I(S_g^n,P)) \geq 1 +\cd(\I(S_g^{n-1},P'))$.

\begin{figure}
    \centering
    \includegraphics[scale = 0.5]{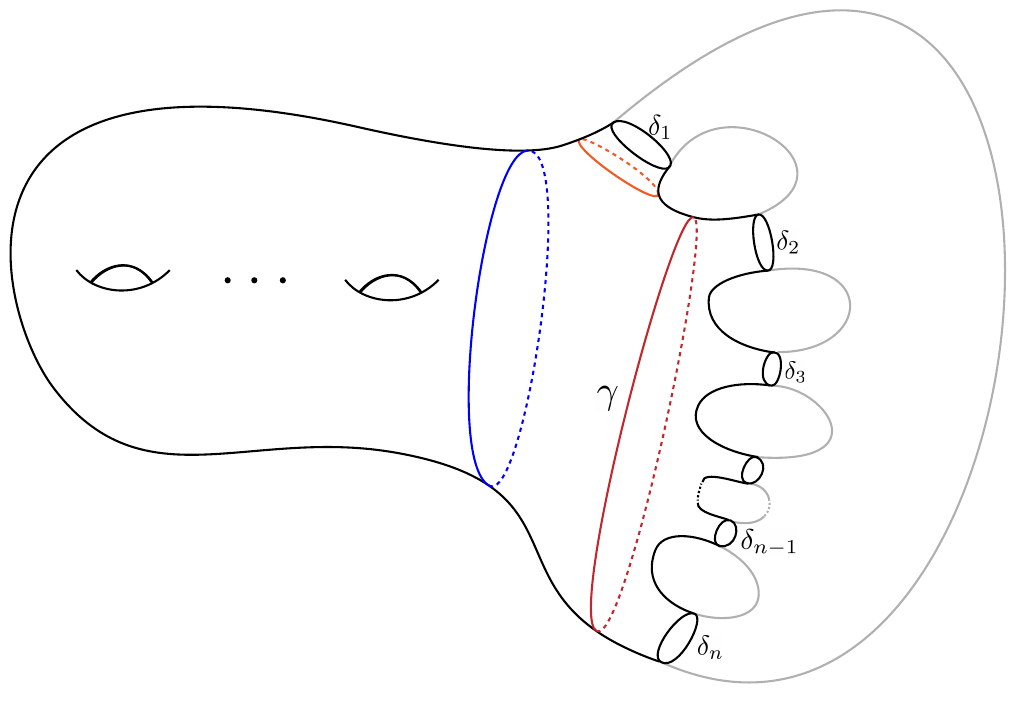}
    \caption{$S_g^n$ is the surface delimited by the black line and one capping is suggested by the shaded lines. The blue curve indicates how we embed $S$ into $S_g^n. $}
    \label{lower lower bound}
\end{figure}

%
%
\medskip
\section{Lower bound on the minimal asymptotic translation length} \label{sec:lowerboundlength}
Let $S$ be a surface of finite type with nonempty boundary and $\beta$ be one of its boundary components. 
Suppose $S'$ is the surface obtained from $S$ by capping the 
boundary component $\beta$ with a once-puncture disk. Let us denote this puncture by $p_0$.
Let $\Mod(S, \{ p_1, \cdots, p_k \})$ be the subgroup of $\Mod(S)$ consisting of
elements that fix the punctures $p_1, \cdots, p_k$ of $S$ (the set $\{p_i\}$ is possibly empty). 
Then we have a homomorphism 
$$\varphi: \Mod(S, \{ p_1, \cdots, p_k \})  \ra \Mod(S', \{ p_0, p_1, \cdots, p_k \})$$
and the following sequence is exact (see Proposition 3.19 in \cite{FarbMargalit12}):
$$ 1 \ra \langle T_{\beta} \rangle \ra \Mod(S, \{ p_1, \cdots, p_k \})  \xrightarrow{\varphi} \Mod(S', \{ p_0, p_1, \cdots, p_k \}) \ra 1.$$
As a consequence, we have the following proposition.

\begin{prop} \label{prop:intopmod} 
There is a homomorphism 
$$\Phi: \Mod(S_g^n) \ra \PMod(S_{g,n}),$$
by capping each boundary component with a once-punctured disk.
\end{prop}

\begin{proof}
By applying Proposition 3.19 in \cite{FarbMargalit12} repeatedly, we obtain a 
homomorphism $\Phi: \Mod(S_g^n) \ra \Mod(S_{g,n})$.
Since the capping homomorphism fix the puncture, the image of $\Phi$ fixes all punctures of $S_{g,n}$.
Therefore, it is contained in $\PMod(S_{g,n})$. 
\end{proof}

Since $S_{g}^n \setminus \partial S_g^n$ is homeomorphic
to $S_{g,n}$ and essential curves are non-peripheral,
$\mathcal{C}(S_{g}^n)$ and $\mathcal{C}(S_{g,n})$ are
naturally identified, and the identification map is
$\Phi$-equivariant. 
This is because the subgroup of $\Mod(S_g^n)$ acting trivially on  $\mathcal{C}(S_{g}^n)$
is generated by the Dehn twists along peripheral curves, and this subgroup is contained in the kernel of $\Phi$. 
Abusing the notation, we will use $\mathcal{C}_{g,n}$ for this curve complex (up to identification).
By the $\Phi$-equivariance, for each $f \in \Mod(S_g^n)$, the actions of $f$ and $\Phi(f)$ on $\C_{g,n}$ are the same. From this, 
we immediately conclude that $\ell_{\C}(f) = \ell_{\C}(\Phi(f)).$

Before we obtain a lower bound on $L_{\C}(\I(S_g^n, P))$, we will prove a lemma which gives us a comparison 
between the groups $\I(S_g^n, P)$ for various partitions $P$ of the set of boundary components.  
Let $P, P'$ be partitions of the set of boundary components. We say $P'$ is finer than $P$ if each element of $P$ is either an element of $P'$ or a union of elements of $P'$. 

\begin{lem} \label{lem:comparison}
Let $P_1, P_2$ be partitions of the set of boundary components of $S_g^n$. Suppose $P_2$ is finer than $P_1$. Then we have $\I(S_g^n, P_1) \subset \I(S_g^n, P_2)$. 
\end{lem} 
\begin{proof}
Let $j_i: (S_g^n,P_i) \xhookrightarrow{} S_i$ be a capping for each $i = 1, 2$. We may assume that each connected component of $S_i \setminus j(S_g^n)$ has no genus. 
Then $\I(S_g^n, P_i) = (j_i)_*^{-1}(\I(S_i))$ by definition (here $S_i$ are closed surfaces so $\I(S_i)$ are usual Torelli groups). 

If $P_2$ is finer than $P_1$ then $S_1$ has ``more" simple closed curves than $S_2$. Since the primitive first homology classes of a surface can be realized by simple closed curves, 
this shows that $ (j_1)_*^{-1}(\I(S_1)) \subset (j_2)_*^{-1}(\I(S_2))$. This proves the lemma. 
\end{proof}

\begin{thm}
For any partitioned surface $(S_g^n,P)$, we have
$$L_{\C}(\I(S_g^n, P)) \gtrsim \frac{1}{|\chi(S_g^n)|}.$$
\end{thm}

\begin{proof} 
Let $f \in \I(S_g^n, P)$. 
By Lemma \ref{lem:comparison}, this implies that $f \in \I(S_g^n, P_{max})$ (as pointed out in \cite{putman2007cutting}, the idea of $\I(S_g^n, P_{max})$ appears in \cite{hain1995torelli}). 
We have a capping $i: (S_g^n, P_{max}) \to S_g$, hence $f \in i_*^{-1}(\I(S_g))$. 

Let $\Phi$ be the map in Proposition \ref{prop:intopmod}. 
Once we identify $S_{g,n}$ with the surface obtained from $S_g$ by removing one point from each connected component of $S_g \setminus i(S_g^n)$, 
$i$ can be seen as the composition $S_g^n \xhookrightarrow{j} S_{g,n} \xhookrightarrow{h} S_g$. 
One can consider a map $j_*: \Mod(S_g^n) \to \Mod(S_{g,n})$ by extending the map as identity outside the image under $j$, and then it would be the same map as $\Phi$.
Therefore the image under $j_*$ is contained in $\PMod(S_{g,n})$. Similarly one can define a map $h_*: \PMod(S_{g,n}) \to \Mod(S_g)$. 

Note that $j_*(f) = \Phi(f) \in \PMod(S_{g,n})$ and $h_*(\Phi(f)) = h_*(j_*(f)) = i_*(f)$. 
Since $f \in i_*^{-1}(\I(S_g))$, this implies that the Lefschetz number $L( \Phi(f) ) := L(h_*(\Phi(f))) = 2- 2g < 0$. 
By Tsai \cite{Tsai09}, $L(\Phi(f))<0$ implies that there exists a rectangle in the Markov partition for $\Phi(f)$ which intersects with its image under $\Phi(f)$.

Finally by \cite{GadreTsai11} (see also \cite{BaikShin18}),
$\ell_C(\Phi(f)) \gtrsim \frac{1}{|\chi(S_{g,n})|}$ and this implies $\ell_C(f) \gtrsim \frac{1}{|\chi(S_g^n)|}$ (recall the discussion right before Lemma \ref{lem:comparison}). 
This completes the proof.
\end{proof}

%
%
\medskip
\section{Upper bound on the minimal asymptotic translation length} \label{sec:upperboundlength}

In this section we obtain the upper bound on $L_{\C}(\I(S_g^n, P))$. We first focus on the case 
 $P = P_{\max}$. To do this, we use Penner's construction to find a pseudo-Anosov element $f \in \I(S_g^n,P_{\max})$ such that 
$ \ell_{\C}(f) \lesssim 1/|\chi(S_g^n)|$.



\begin{thm} \label{thm:maxpartitioncase}
We have
$$L_{\C}(\I(S_g^n, P_{\max})) \lesssim \frac{1}{|\chi(S_g^n)|}.$$
\end{thm}

\begin{proof} For notational simplicity, we simply write $P$ for $P_{\max}$ throughout the proof.  

\begin{figure}[t]	
\centering
\begin{tikzpicture}[scale=.6]		
		\foreach \z in {0,-6} {
		\draw[thick] (0,0) to [out=90, in=180] (7,2);
		\draw[thick] (14.5,0) to [out=90, in=0] (7,2);
		\draw[thick] (0,0) to [out=-90, in=180] (7,-2);
		\draw[thick] (14.5,0) to [out=-90, in=0] (7,-2);    
		\draw[thick] (0,0-6) to [out=90, in=180] (8,2-6);
		\draw[thick] (16.3,0-6) to [out=90, in=0] (8,2-6);
		\draw[thick] (0,0-6) to [out=-90, in=180] (8,-2-6);
		\draw[thick] (16.3,0-6) to [out=-90, in=0] (8,-2-6);  
        
		\node at (7, -2.8) {(a) $g=2+2m,$ even};
		\node at (7, -9) {(b) $g=3+2m$, odd};
		\node at (0.6, 2.2) {$S_g^n$};
				
		\node[blue] at (.7,-.5+\z) {\footnotesize $a_k$}; 
		\node[blue] at (2.75,-.7+\z) {\footnotesize $a_2$}; 
		\node[red] at (3.5,-.7+\z) {\footnotesize $a_1$}; 
		
		\node[red] at (6,-.75+\z) {\footnotesize $b_1$}; 
		\node[blue] at (7,-.75+\z) {\footnotesize $b_2$}; 
		\node[red] at (8.5,-.75+\z) {\footnotesize $b_l$}; 
		
		\node[red] at (11,-.7) {\footnotesize $a_{k+1}$}; \node[red] at (11,-.7-6) {\footnotesize $a_{k+1}$};
		\node[blue] at (13.2,-.7) {\footnotesize $a_{g-2}$};
		\node[red] at (13,-.9-6) {\footnotesize $a_{g-4}$};
		\node[blue] at (14.1,-1.2-6) {\footnotesize $a_{g-3}$};
		\draw[blue,->] (13.9,-1-6) to (13.9,-.45-6);
		\node[red] at (15.2,-.7-6) {\footnotesize $a_{g-2}$};
		\node[blue] at (14.1,.8-6) {\footnotesize $a_{g-1}$};
		
        \node[blue] at (4.8,-.9+\z) {\footnotesize $c_1$};
        \node[blue] at (9.8,-.9+\z) {\footnotesize $c_2$};
        \node[red] at (8,-7.5) {\footnotesize $c_3$};
        \node[red] at (13,-2.5) {\footnotesize $c_3$};\draw[red,->] (12.9,-2.3) to (12.3,-1.1);
        \node[red] at (14,-2.5) {\footnotesize $c_4$};\draw[red,->] (13.9,-2.3) to (13.4,-1.3);
        
		\node[blue] at (2.25,-.7+\z) {\footnotesize $\cdots$};
		\node[blue] at (7.7,-.75+\z) {\footnotesize $\cdots$};
		\node[blue] at (12,-.75+\z) {\footnotesize $\cdots$};

        \foreach \x in {1,3,4,10,11,13}{
        \draw[thick] (\x,-.1+\z) .. controls (.15+\x,0+\z) and (.35+\x,0+\z) .. (.5+\x,-.1+\z);
        \draw[thick] (-0.04+\x,-.09+\z) .. controls (.15+\x,-.2+\z) and (.35+\x,-.2+\z) .. (.54+\x,-.09+\z);
        }
        \foreach \x in {14,15}{
        \draw[thick] (\x,-.1-6) .. controls (.15+\x,0-6) and (.35+\x,0-6) .. (.5+\x,-.1-6);
        \draw[thick] (-0.04+\x,-.09-6) .. controls (.15+\x,-.2-6) and (.35+\x,-.2-6) .. (.54+\x,-.09-6);
        }
        
		\foreach \x in {5,6,7.5,8.5}{
		\draw[thick] (\x+.5,-.1+\z) circle [radius=.15];
		}        
		\draw[thick] (.9,-1.2+\z) circle [radius=.15];
		\node at (1.2, -1.4+\z) {\footnotesize $\beta$};
        
        \foreach \x in {2,7,12}{
        \node at (\x+0.4,-.1+\z) {$\cdots$};
        }
        
        \draw [red,semithick] plot [smooth cycle, tension=1] coordinates {(5.1,-.1+\z) (6,.3+\z) (6.9,-.1+\z) (6,-.5+\z)};
        \draw [blue,semithick] plot [smooth, tension=2] coordinates {(7,.3+\z) (6.1,-.1+\z) (7,-.5+\z)};
        \draw [blue,semithick] plot [smooth, tension=2] coordinates {(7.5,.3+\z) (8.4,-.1+\z) (7.5,-.5+\z)};
        \draw [red,semithick] plot [smooth cycle, tension=1] coordinates {(7.6,-.1+\z) (8.5,.3+\z) (9.4,-.1+\z) (8.5,-.5+\z)};
        
		\draw [blue,semithick] plot [smooth, tension=3.1] coordinates {(3.6,2.1+0) (4.8,-.65+0) (5.8,2.02+0)};
		\draw [blue,semithick,dotted] plot [smooth, tension=2.9] coordinates {(3.6,2.1+0) (4.8,-.6+0) (5.8,2.02+0)};
		
		\draw [blue,semithick] plot [smooth, tension=3.1] coordinates {(3.6+4.9,2.05+0) (4.8+4.9,-.65+0) (5.8+4.9,2.16+0)};
		\draw [blue,semithick,dotted] plot [smooth, tension=2.9] coordinates {(3.6+4.9,2.02+0) (4.8+4.9,-.6+0) (5.8+4.9,2.1+0)};
		
		\draw [blue,semithick] plot [smooth, tension=3.1] coordinates {(3.6,2.2-6) (4.8,-.65-6) (5.8,2.1-6)};
		\draw [blue,semithick,dotted] plot [smooth, tension=2.9] coordinates {(3.6,2.1-6) (4.8,-.6-6) (5.8,2.02-6)};
		
		\draw [blue,semithick] plot [smooth, tension=3.1] coordinates {(3.6+4.9,2.-6) (4.8+4.9,-.65-6) (5.8+4.9,2.16-6)};
		\draw [blue,semithick,dotted] plot [smooth, tension=2.9] coordinates {(3.6+4.9,2.02-6) (4.8+4.9,-.6-6) (5.8+4.9,2.1-6)};

        
        \foreach \y in {-1}{
        \foreach \x in {.7}{
        \draw[semithick,blue] (\x,.9+\y+\z) to [out=90,in=130] (1.3+\x,.94+\y+\z);
        \draw[semithick,blue] (\x,.9+\y+\z) to [out=-90,in=-130] (1.3+\x,.86+\y+\z);
        \draw[semithick,blue,dotted] (.1+\x,.9+\y+\z) to [out=90,in=140] (1.3+\x,.94+\y+\z);
        \draw[semithick,blue,dotted] (.1+\x,.9+\y+\z) to [out=-90,in=-140] (1.3+\x,.86+\y+\z);
		}}
		
		\foreach \y in {-1}{
        \foreach \x in {1.7}{
        \draw[semithick,blue] (.8+\x,1.2+\y+\z) to [out=-10,in=130] (1.3+\x,.94+\y+\z);
        \draw[semithick,blue] (.8+\x,.6+\y+\z) to [out=10,in=-130] (1.3+\x,.86+\y+\z);
		}}
		
		\foreach \y in {-1}{
        \foreach \x in {-13.75}{
        \draw[semithick,blue] (-\x,.9+\y) to [out=90,in=50] (-1.3-\x,.94+\y);
        \draw[semithick,blue] (-\x,.9+\y) to [out=-90,in=-50] (-1.3-\x,.86+\y);
        \draw[semithick,blue,dotted] (-.1-\x,.9+\y) to [out=90,in=40] (-1.3-\x,.94+\y);
        \draw[semithick,blue,dotted] (-.1-\x,.9+\y) to [out=-90,in=-40] (-1.3-\x,.86+\y);
		}}
		
		\foreach \y in {-1-6}{
        \foreach \x in {-14.75}{
        \draw[semithick,blue] (-\x,.9+\y) to [out=90,in=50] (-1.3-\x,.94+\y);
        \draw[semithick,blue] (-\x,.9+\y) to [out=-90,in=-50] (-1.3-\x,.86+\y);
        \draw[semithick,blue,dotted] (-.1-\x,.9+\y) to [out=90,in=40] (-1.3-\x,.94+\y);
        \draw[semithick,blue,dotted] (-.1-\x,.9+\y) to [out=-90,in=-40] (-1.3-\x,.86+\y);
		}}
		
		\foreach \y in {0}{
        \foreach \x in {-12.5}{
        \draw[semithick,red] (-\x,.25+\y) to [out=-10,in=10] (-\x,-.4+\y);
		}}        
		\foreach \y in {0}{
        \foreach \x in {-12.5}{
        \draw[semithick,blue] (-\x,.25+\y-6) to [out=-10,in=10] (-\x,-.4+\y-6);
		}}        
        
        \foreach \y in {-1}{
        \foreach \x in {13.7}{
        \draw[semithick,blue] (\x-.9,.9+\y-6) to [out=90,in=120] (1.3+\x,.94+\y-6);
        \draw[semithick,blue] (\x-.9,.9+\y-6) to [out=-90,in=-120] (1.3+\x,.86+\y-6);
        \draw[semithick,blue,dotted] (.1+\x-.9,.9+\y-6) to [out=90,in=130] (1.3+\x,.94+\y-6);
        \draw[semithick,blue,dotted] (.1+\x-.9,.9+\y-6) to [out=-90,in=-130] (1.3+\x,.86+\y-6);
		}}

		\foreach \y in {0}{
        \foreach \x in {2}{
        \draw[semithick,red] (\x,.25+\y+\z) to [out=190,in=170] (\x,-.4+\y+\z);
		}}        
        
        \foreach \y in {-1}{
        \foreach \x in {2.7}{
        \draw[semithick,red] (\x,.9+\y+\z) to [out=90,in=130] (1.3+\x,.94+\y+\z);
        \draw[semithick,red] (\x,.9+\y+\z) to [out=-90,in=-130] (1.3+\x,.86+\y+\z);
        \draw[semithick,red,dotted] (.1+\x,.9+\y+\z) to [out=90,in=140] (1.3+\x,.94+\y+\z);
        \draw[semithick,red,dotted] (.1+\x,.9+\y+\z) to [out=-90,in=-140] (1.3+\x,.86+\y+\z);
		}}
		
		\foreach \y in {-1}{
        \foreach \x in {-11.75}{
        \draw[semithick,red] (-\x,.9+\y+\z) to [out=90,in=50] (-1.3-\x,.94+\y+\z);
        \draw[semithick,red] (-\x,.9+\y+\z) to [out=-90,in=-50] (-1.3-\x,.86+\y+\z);
        \draw[semithick,red,dotted] (-.1-\x,.9+\y+\z) to [out=90,in=40] (-1.3-\x,.94+\y+\z);
        \draw[semithick,red,dotted] (-.1-\x,.9+\y+\z) to [out=-90,in=-40] (-1.3-\x,.86+\y+\z);
		}}
		\foreach \y in {-1}{
        \foreach \x in {-13.75,-15.75}{
        \draw[semithick,red] (-\x,.9+\y-6) to [out=90,in=50] (-1.3-\x,.94+\y-6);
        \draw[semithick,red] (-\x,.9+\y-6) to [out=-90,in=-50] (-1.3-\x,.86+\y-6);
        \draw[semithick,red,dotted] (-.1-\x,.9+\y-6) to [out=90,in=40] (-1.3-\x,.94+\y-6);
        \draw[semithick,red,dotted] (-.1-\x,.9+\y-6) to [out=-90,in=-40] (-1.3-\x,.86+\y-6);
		}}

		\foreach \y in {-1}{
        \foreach \x in {-12.75}{
        \draw[semithick,blue] (-.8-\x,1.2+\y+\z) to [out=180+10,in=50] (-1.3-\x,.94+\y+\z);
        \draw[semithick,blue] (-.8-\x,.6+\y+\z) to [out=170,in=-50] (-1.3-\x,.86+\y+\z);
		}}
		
		\draw [red,semithick] plot [smooth, tension=.5] coordinates {(13.1,-.04) (11,1) (3,1) (1.65,-.08) (3,-1.3) (11,-1.3) (13.1,-0.15)};
		\draw [red,semithick,dotted] plot [smooth, tension=.55] coordinates {(13.1,-.04) (11,1.05) (3,1.05) (1.6,-.08) (3,-1.35) (11,-1.35) (13.1,-0.15)};
		\draw [red,semithick] plot [smooth, tension=.5] coordinates {(1.2,-.04) (3,1.3) (12,1.3) (14,-0.08) (12,-1.6) (3,-1.6) (1.2,-0.15)};
		\draw [red,semithick,dotted] plot [smooth, tension=.55] coordinates {(1.2,-.04) (3,1.35) (12,1.35) (14.1,-0.08) (12,-1.65) (3,-1.65) (1.2,-0.15)};
		\draw [red,semithick] plot [smooth, tension=.5] coordinates {(1.2,-.04-6) (3,1.3-6) (14,1.3-6) (16,-0.08-6) (14,-1.6-6) (3,-1.6-6) (1.2,-0.15-6)};
		\draw [red,semithick,dotted] plot [smooth, tension=.55] coordinates {(1.2,-.04-6) (3,1.35-6) (14,1.35-6) (16.1,-0.08-6) (14,-1.65-6) (3,-1.65-6) (1.2,-0.15-6)};
        
        }

\end{tikzpicture}
\caption{Red curves and blue curves are two filling multi-curves of $S_g^n$, 
consisting of separating curves. If the number of boundary components $n$ is odd, 
$\beta$ is a boundary component. If $n$ is even, we consider $\beta$ capped by 
a disk and hence it is not a boundary component. The numbers $k$ and $l$ are 
defined by $\displaystyle k=m+\frac{1+(-1)^{m+1}}{2}$ and $\displaystyle l=n-1-\frac{1+(-1)^{n+1}}{2}$.}
\label{figure:upperbound}
\end{figure}
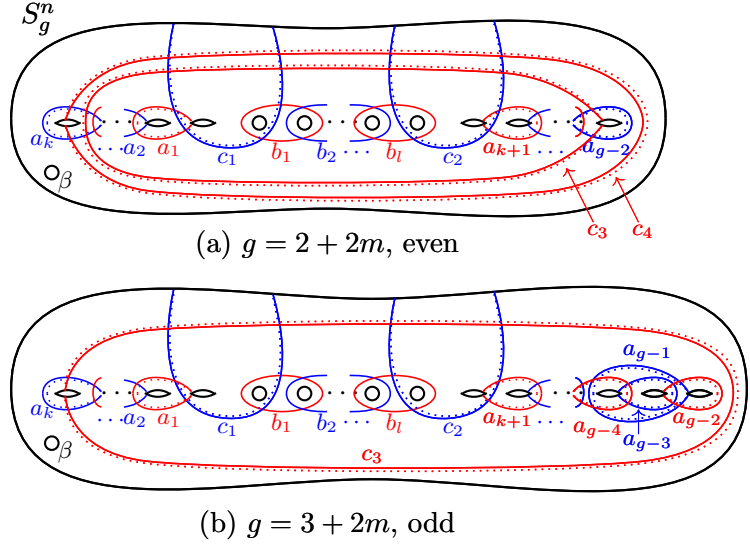

It suffices to show that there is a pseudo-Anosov mapping class $f \in \I(S_g^n, P)$ such that
$ \ell_{\C}(f) \lesssim 1/|\chi(S_g^n)|$.
The idea is as follows. We will find filling multicurves $A$ and $B$ consisting of separating curves.  
Then define $f = T_A T_B^{-1}$, where $T_A$ and $T_B$ are multi-twists. 
Then $f$ is pseudo-Anosov since it comes from Penner's construction.

Since $P$ is maximal, every separating curve on $S_g^n$ is still separating on the capped surface, hence the Dehn twist long a separating curve is indeed in $\I(S_g^n, P)$. This implies that $f$ lies in $\I(S_g^n, P)$. 

We then find a simple closed curve $\gamma$ such that $f^N(\gamma)$ and $\gamma$ 
don't fill the surface together for some integer $N$.
This implies that $d_{\C}(f^N(\delta), \delta) \leq 2$ and further that
$$\ell_{\C}(f) \leq \frac{2}{N}.$$
We will consider when genus $g$ is even and when $g$ is odd, separately.

\medskip
\textbf{Case 1.} Assume $g=2+2m$ for some $m$.
Let $A$ consist of red curves and $B$ consist of blue curves as in Figure \ref{figure:upperbound}(a).
In this figure, if the number of boundary components $n$ is odd, then we place $n-1$ boundary components 
in the middle and put the last boundary component, say $\beta$, on the left bottom corner.
If $n$ is even, we place all boundary components in the middle and $\beta$ doesn't represent the boundary component.
Then the numbers $k$ and $l$ are defined by 
$$ k=m+\frac{1+(-1)^{m+1}}{2} \quad \textrm{and} \quad  l=n-1-\frac{1+(-1)^{n+1}}{2}.$$
Note that $k$ is always even and $l$ is always odd. Define $f = T_A T_B^{-1}$.

\begin{figure}[t]	
\centering
\begin{tikzpicture}[scale=.7]		
		
		\foreach \x in {0,...,7}{
		\draw[thick] (\x,0) circle [radius=.15];
		} 
		
		\foreach \y in {0}{
        \foreach \x in {0,2,4,6}{
        \draw [red,semithick] plot [smooth cycle, tension=1] coordinates {(-.3+\x,0) (.5+\x,.5) (1.3+\x,0) (.5+\x,-.5)};
		}}   
		
		\foreach \y in {0}{
        \foreach \x in {1,3,5}{
        \draw [blue,semithick] plot [smooth cycle, tension=1] coordinates {(-.3+\x,0) (.5+\x,.5) (1.3+\x,0) (.5+\x,-.5)};
		}}
		
		\draw [blue,semithick] plot [smooth, tension=1] coordinates {(.5,1) (.2,-.3) (-.3,-.6)};
		\draw [blue,semithick] plot [smooth, tension=1] coordinates {(6.5,1) (6.8,-.3) (7.3, -.6)};
		
		\foreach \x in {3}{
		\draw [black!30!green,thick] plot [smooth cycle, tension=1] coordinates {(-.3+\x,0) (.5+\x,.5) (1.3+\x,0) (.5+\x,-.5)};
		}
		
		\node[black!30!green] at (3.5,.8) {\footnotesize $\delta$};
		\node[blue] at (3.5,-.8) {\footnotesize $b_{\lceil \frac{l}{2} \rceil}$};
		\node at (2,-.8) {\footnotesize $\cdots$};
		\node[red] at (.5, -.8) {\footnotesize $b_1$};
		\node at (3,-2) {(c) When $\lfloor n/2 \rfloor$ is even};
		
		\foreach \z in {4.5}{
		\foreach \x in {0,...,9}{
		\draw[thick] (\x,0-\z) circle [radius=.15];
		} 
		
		\foreach \y in {0}{
        \foreach \x in {0,2,4,6,8}{
        \draw [red,semithick] plot [smooth cycle, tension=1] coordinates {(-.3+\x,0-\z) (.5+\x,.5-\z) (1.3+\x,0-\z) (.5+\x,-.5-\z)};
		}}   
		
		\foreach \y in {0}{
        \foreach \x in {1,3,5,7}{
        \draw [blue,semithick] plot [smooth cycle, tension=1] coordinates {(-.3+\x,0-\z) (.5+\x,.5-\z) (1.3+\x,0-\z) (.5+\x,-.5-\z)};
		}}
		
		\draw [blue,semithick] plot [smooth, tension=1] coordinates {(.5,1-\z) (.2,-.3-\z) (-.3,-.6-\z)};
		\draw [blue,semithick] plot [smooth, tension=1] coordinates {(8.5,1-\z) (8.8,-.3-\z) (9.3, -.6-\z)};
		
		\foreach \x in {4}{
		\draw [black!30!green,thick] plot [smooth cycle, tension=1] coordinates {(-.3+\x,0-\z) (.5+\x,.5-\z) (1.3+\x,0-\z) (.5+\x,-.5-\z)};
		}
		
		\node[black!30!green] at (4.5,.8-\z) {\footnotesize $\delta$};
		\node[red] at (4.5,-.8-\z) {\footnotesize $b_{\lceil \frac{l}{2} \rceil}$};
		\node at (2.5,-.8-\z) {\footnotesize $\cdots$};
		\node[red] at (.5, -.8-\z) {\footnotesize $b_1$};
		\node at (3,-2-\z) {(d) When $\lfloor n/2 \rfloor$ is odd};
		 }


		 \foreach \z in {0,-4.5}{
		 
		  \draw[thick] (-2.8,\z-.1) circle [radius=.15];

        \foreach \x in {-4, -5, -6, -8}{
        \draw[thick] (\x,-.1+\z) .. controls (.15+\x,0+\z) and (.35+\x,0+\z) .. (.5+\x,-.1+\z);
        \draw[thick] (-0.04+\x,-.09+\z) .. controls (.15+\x,-.2+\z) and (.35+\x,-.2+\z) .. (.54+\x,-.09+\z);
        }
        
        \node at (-6.6, \z-.1) {$\cdots$};
        
        \foreach \y in {-1}{
        \foreach \x in {-5.3}{
        \draw[semithick,red] (\x,.9+\y+\z) to [out=90,in=130] (1.35+\x,.94+\y+\z);
        \draw[semithick,red] (\x,.9+\y+\z) to [out=-90,in=-130] (1.35+\x,.86+\y+\z);
        \draw[semithick,red,dotted] (.1+\x,.9+\y+\z) to [out=90,in=140] (1.35+\x,.94+\y+\z);
        \draw[semithick,red,dotted] (.1+\x,.9+\y+\z) to [out=-90,in=-140] (1.35+\x,.86+\y+\z);
		}}
		\foreach \y in {-1}{
        \foreach \x in {-6.3}{
        \draw[semithick,blue] (\x,.9+\y+\z) to [out=90,in=130] (1.35+\x,.94+\y+\z);
        \draw[semithick,blue] (\x,.9+\y+\z) to [out=-90,in=-130] (1.35+\x,.86+\y+\z);
        \draw[semithick,blue,dotted] (.1+\x,.9+\y+\z) to [out=90,in=140] (1.35+\x,.94+\y+\z);
        \draw[semithick,blue,dotted] (.1+\x,.9+\y+\z) to [out=-90,in=-140] (1.35+\x,.86+\y+\z);
		}}

		\draw [red,semithick] plot [smooth, tension=2] coordinates {(-1.9,.5+\z) (-3.2, -.2+\z) (-1.9, -.9+\z)};
		\draw [blue,semithick] plot [smooth, tension=2] coordinates {(-4.5,1+\z) (-3.2, -.6+\z) (-2, 1+\z)};
		
		\draw [black!30!green,thick] plot [smooth cycle, tension=1] coordinates {(-8.1,-.1+\z) (-7.7,.15+\z) (-7.3,-.1+\z) (-7.7,-.35+\z) };
		\node[black!30!green] at (-7.8,.5+\z) {\footnotesize $\gamma$};
		
		\node[red] at (-4.5, -.8+\z) {\footnotesize $a_1$};
		\node[blue] at (-5.5, -.8+\z) {\footnotesize $a_2$};
		\node[blue] at (-3.6, -.8+\z) {\footnotesize $c_1$};
		\node at (-6.5,-.8+\z) {\footnotesize $\cdots$};
		}
		
		\node[red] at (-7.5, -.8-4.5) {\footnotesize $a_{\frac{k}{2}}$};
		\node[blue] at (-7.5, -.8) {\footnotesize $a_{\frac{k}{2}}$};
				
		\foreach \z in {-4.5}{
		\foreach \y in {-1}{
        \foreach \x in {-8.3}{
        \draw[semithick,red] (\x,.9+\y+\z) to [out=90,in=130] (1.35+\x,.94+\y+\z);
        \draw[semithick,red] (\x,.9+\y+\z) to [out=-90,in=-130] (1.35+\x,.86+\y+\z);
        \draw[semithick,red,dotted] (.1+\x,.9+\y+\z) to [out=90,in=140] (1.35+\x,.94+\y+\z);
        \draw[semithick,red,dotted] (.1+\x,.9+\y+\z) to [out=-90,in=-140] (1.35+\x,.86+\y+\z);
		}}}
		\foreach \z in {0}{
		\foreach \y in {-1}{
        \foreach \x in {-8.3}{
        \draw[semithick,blue] (\x,.9+\y+\z) to [out=90,in=130] (1.35+\x,.94+\y+\z);
        \draw[semithick,blue] (\x,.9+\y+\z) to [out=-90,in=-130] (1.35+\x,.86+\y+\z);
        \draw[semithick,blue,dotted] (.1+\x,.9+\y+\z) to [out=90,in=140] (1.35+\x,.94+\y+\z);
        \draw[semithick,blue,dotted] (.1+\x,.9+\y+\z) to [out=-90,in=-140] (1.35+\x,.86+\y+\z);
		}}}
		
		\node at (-5,-2) {(a) When $k/2$ is even};
		\node at (-5,-2-4.5) {(b) When $k/2$ is odd};

\end{tikzpicture}
\caption{The choice of $\gamma$ and $\delta$} \label{figure:gammadelta}
\end{figure}
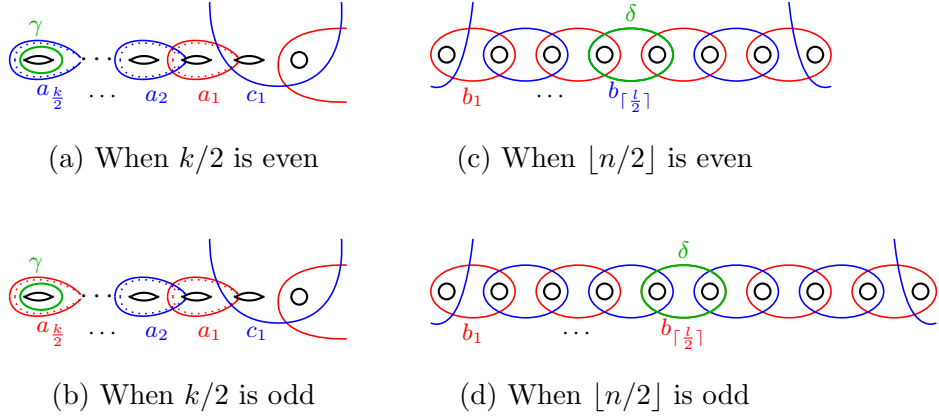

Now we follow the same notation as in the proof of Theorem 6.1 in \cite{GadreTsai11}.
For a finite collection of curves $\{c_j \}_{j=1}^{k}  \subset A \cup B$ such that
$c_1 \cup \cdots \cup c_k$ is connected, 
let $\NN(c_1 \cdots c_k)$ be the regular neighborhood of $c_1 \cup \cdots \cup c_k$.
Let $\delta$ be the curve $b_{\lceil \frac{l}{2} \rceil}$ and let $\gamma$ be the curve 
as in Figure \ref{figure:gammadelta}.
Recall that $f = T_A T_B^{-1}$ and consider the images of $f$.
One can see that
\begin{align*}
& f(\delta) \subset 
\left\{ \begin{array}{ll}
\NN(b_{\lceil \frac{l}{2} \rceil-1} b_{\lceil \frac{l}{2} \rceil} b_{\lceil \frac{l}{2} \rceil+1}), & \textrm{if $\lfloor \frac{n}{2} \rfloor$ is even,}\\
\NN(b_{\lceil \frac{l}{2} \rceil-2} b_{\lceil \frac{l}{2} \rceil-1} b_{\lceil \frac{l}{2} \rceil} b_{\lceil \frac{l}{2} \rceil+1} b_{\lceil \frac{l}{2} \rceil+2}), & \textrm{if $\lfloor \frac{n}{2} \rfloor$ is odd,}
\end{array} \right.
\\
\\
& f^{2}(\delta) \subset 
\left\{ \begin{array}{ll}
\NN(b_{\lceil \frac{l}{2} \rceil-3}  \cdots  b_{\lceil \frac{l}{2} \rceil+3}), & \hspace{5.2em} \textrm{if $\lfloor \frac{n}{2} \rfloor$ is even,}\\
\NN(b_{\lceil \frac{l}{2} \rceil-4}  \cdots  b_{\lceil \frac{l}{2} \rceil+4}), & \hspace{5.2em }\textrm{if $\lfloor \frac{n}{2} \rfloor$ is odd,}
\end{array} \right.
\\
&\vdots\\
& f^{\lfloor \frac{n}{4} \rfloor} ( \delta)  \subset \NN( b_1 \cdots b_l)\\
&\vdots\\
& f^{\lceil \frac{k}{4} \rceil + \lfloor \frac{n}{4} \rfloor} ( \delta)  \subset \NN( a_{t} \cdots a_{k+t}), \qquad \textrm{where $t=2 \Big\lceil \frac{k}{4} \Big\rceil -1$.}
\end{align*}
Note that $\gamma$ can be realized as not included $\NN( a_{t} \cdots a_{k+t})$ and hence 
$\gamma$ doesn't intersect $f^{\lceil \frac{k}{4} \rceil + \lfloor \frac{n}{4} \rfloor} ( \delta)$.
This implies that $f^{\lceil \frac{k}{4} \rceil + \lfloor \frac{n}{4} \rfloor} (\delta)$ and $\delta$ doesn't fill the surface,
i.e., $d_{\C}(f^{\lceil \frac{k}{4} \rceil + \lfloor \frac{n}{4} \rfloor}(\delta), \delta) \leq 2$. Hence we have
\begin{align*}
	\ell_{\C}(f) &\leq \frac{2}{\lceil \frac{k}{4} \rceil + \lfloor \frac{n}{4} \rfloor} \leq \frac{2}{\frac{k}{4}+\frac{n}{4}-1}\\
					&\leq \frac{8}{k+n-4} = \frac{8}{\frac{g-2}{2} + n -4}\\
					&\leq \frac{16}{g+2n-10} < \frac{16}{g+\frac{n}{2}-10}\\
					&= \frac{32}{|\chi(S_g^n)|-18}.
\end{align*}
Therefore, $\displaystyle \ell_{\C}(f) \lesssim \frac{1}{|\chi(S_g^n)|}.$

\vspace{.5em}
\textbf{Case 2.} Assume $g=3+2m$ for some $m$.

In this case, the computation is almost identical to Case 1.
Let $A$ consist of red curves and $B$ consist of blue curves as in Figure \ref{figure:upperbound}(b).
Here, the numbers $k$ and $l$ are defined by 
$$ k=m+\frac{1+(-1)^{m+1}}{2} \quad \textrm{and} \quad  l=n-1-\frac{1+(-1)^{n+1}}{2}.$$
Define $f = T_A T_B^{-1}$ and choose $\gamma$ and $\delta$ as in Figure \ref{figure:gammadelta}.
Then by the same argument as in Case 1, we conclude that $\gamma$ doesn't intersect $f^{\lceil \frac{k}{4} \rceil + \lfloor \frac{n}{4} \rfloor} ( \delta)$. 
Then $f^{\lceil \frac{k}{4} \rceil + \lfloor \frac{n}{4} \rfloor} (\delta)$ and $\delta$ doesn't fill the surface, and hence we have
again that
\begin{align*}
	\ell_{\C}(f) &\leq \frac{2}{\lceil \frac{k}{4} \rceil + \lfloor \frac{n}{4} \rfloor} \leq  \frac{32}{|\chi(S_g^n)|-18}.
\end{align*}
Therefore, $\displaystyle \ell_{\C}(f) \lesssim \frac{1}{|\chi(S_g^n)|}.$

\end{proof}

Now let $P$ be a general partition. We obtain a new surface $S'$ from $S_g^n$ as follows: for each element $c$ of $P$ which has more than one boundary components, we cap off each boundary components in $c$ by a disk for all but one boundary components of $c$. Then $S'$ has the same genus as $S$ while the number of boundary component may decrease. Let $P'$ be the induced partition of boundary of $S'$. By the construction, $P'$ is maximal. 

As we mentioned in Section \ref{sec:cohomid}, for any boundary $d$ of $S_g^n$, there is a map $\capd: \I(S_g^n, P) \to \I(S', \Tilde{P})$, where $\Tilde{P}$ is obtained from $P$ by deleting $d$ from the set $p \in P$ containing it. In \cite{putman2007cutting}, this map is proven to be surjective. We can thus obtain a surjective map  $\rho: \I(S_g^n, P) \to \I(S', P')$ by composition of maps of the form $\capd$, one for each of the boundary we want to cap off.
Since $S'$ has less simple closed curves than $S_g^n$, for each $f \in \I(S_g^n, P)$, $\ell(f) \leq \ell(\rho(f))$. This implies that $L_{\C}(\I(S_g^n, P)) \leq L_{\C}(\I(S', P'))$. On the other hand, since $P'$ is maximal, by Theorem \ref{thm:maxpartitioncase}, we have 
$L_{\C}(\I(S', P')) \leq C/|\chi(S')|$. Now observe that the boundary components of $S'$ are one-to-one correspondence with elements of $P$. Hence, $|\chi(S')| + (n-|P|) = |\chi(S_g^n)|$. The upper bound in Theorem \ref{thm:mainthmlength} follows.

%
%
\medskip
\section{Further discussions} \label{sec:future} 

One obvious improvement to theorem \ref{main} would be to find a precise formula for the cohomological dimension of $\I(S_g^n,P)$ for all $P$. In \cite{putman2007cutting}, the author characterizes precisely the kernel of the map $\capd \colon \I(S_g^n,P) \to \I(S_g^{n-1}, P')$. It is either the whole fundamental group of the unitary tangent bundle of $S_g^{n-1}$ when we cap a boundary $d$ such that $\{d\} \in P$, or a graph of the kernel of $\pi_1(S_g^{n-1}) \to H_1^{P'}(S_g^{n-1})$ in all other cases. It would be good to find a similar characterization of the kernel of $\capd_n \colon \I(S_g^n,P) \to \I(S_g)$ and that would possibly improve our upper bound for the cohomological dimension of $\I(S_g^n,P)$. Indeed, we showed that this kernel fits in a short exact sequence 

$$
1 \to \mathbb{Z}^m \to K \to K' \to 1
$$
and used the fact that $K'$ is a subgroup of $\PBG_n(S_g)$ and thus has cohomological dimension at most $k+1$. But a more careful analysis of $K'$ might yield a better estimation of its cohomological dimension.

Recall that the minimal asymptotic translation lengths for $\Mod(S_g)$ and $\Mod(\mathcal{I}_g)$ have different order. More precisely, 
for closed surfaces we have 
$L_{\C}(\Mod(S_g)) \asymp 1/g^2$ by Gadre-Tsai \cite{GadreTsai11} while 
$L_{\C}(\Mod(\mathcal{I}_g)) \asymp 1/g$ by Baik-Shin \cite{BaikShin18}.

One could ask if an analogous statement would be true for either $S_g^n$ or $S_{g,n}$. First of all, Valdivia showed in \cite{Valdivia14} (for (1), see also \cite{BaikShin18}) that 
\begin{enumerate}
	\item When $g$ is fixed, $L_{\C}(\PMod(S_{g,n})) \asymp 1/|\chi(S_{g,n})|$
	\item When $g = rn$ for some $r \in \mathbb{Q}$, $L_{\C}(\Mod(S_{g,n})) \asymp 1/|\chi(S_{g,n})|^2$. 
\end{enumerate} 

As we pointed out before, $L_{\C}(\Mod(S_g^n)) = L_{\C}(\PMod(S_{g,n}))$, since we have a map $\Phi: \Mod(S_g^n) \to  \PMod(S_{g,n})$ so that their actions on the curve complexes is $\Phi$-equivariant. Hence by (1) above, it is not possible to have $L_{\C}(\Mod(S_g^n)) \asymp 1/|\chi(S_g^n)|^2$. 

We conjecture that both phenomena are true in general. 

\begin{conj}\label{conj:punctured}
The followings hold. 
\begin{enumerate}
	\item $L_{\C}(\PMod(S_{g,n})) \asymp 1/|\chi(S_{g,n})|$.
	\item $L_{\C}(\Mod(S_{g,n})) \asymp 1/|\chi(S_{g,n})|^2$. 
\end{enumerate} 
\end{conj}

Valdivia's result is a partial evidence to both parts of Conjecture \ref{conj:punctured}. As an additional partial evidence to Conjecture \ref{conj:punctured} (1), we show the following. 

\begin{thm} Suppose $g < (1/4 - \epsilon) n$ for arbitrarily small $\epsilon > 0$. Then 
$L_{\C}(\PMod(S_{g,n})) \asymp 1/|\chi(S_{g,n})|$.
\end{thm}

\begin{proof} We first remark that it is sufficient to show that
$L_{\C}(\PMod(S_{g,n})) \gtrsim 1/|\chi(S_{g,n})|$ due to the construction given in Section 4 of \cite{Valdivia14}. For this, we roughly follow the proof of Theorem 7.1 (see also the proof of Theorem 5.1) in \cite{BaikShin18}. 

Let $\phi \in \PMod(S_{g,n})$ be a pseudo-Anosov element and  
$\tau$ be its invariant train track obtained from Bestvina-Handel algorithm \cite{BestvinaHandel95}. What we need to know about the train track obtained from Bestvina-Handel algorithm is summarized in Section 2 of 
\cite{BaikShin18}. Here is the list of facts we need. 
\begin{enumerate}
	\item $\tau$ has two types of branches, real and infinitesimal. 
	\item The number of real branches is smaller or equals to $9 |\chi(S_{g,n})|$. 
	\item Let $M_\mathcal{R}$ be the transition matrix for $\phi$ on the real branches of $\tau$. If $q$ is a positive integer so that $M_{\mathcal{R}}^q$ has a positive diagonal entry, then the $\ell_\mathcal{C}(\phi) \gtrsim 1/ q|\chi(S_{g,n})|$ (this follows from Proposition 2.2 of \cite{BaikShin18} which is based on the work of \cite{MasurMinsky99}, \cite{GadreTsai11}, \cite{GadreHironakaKentLeininger13}). 
\end{enumerate}

Hence, it is enough to show that the number $q$ above is uniformly bounded by a constant.

As explained in the proof of Theorem 5.1 in \cite{BaikShin18}, each puncture is contained in a distinct ideal polygons obtained as connected components of the complement of $\tau$. Let $k_1$ be the number of punctures contained in monogons and $k_2$ be the number of punctures contained in bigons. Let $\mathcal{S}$ be the set of singularities of the invariant foliation for $\phi$ whose index is greater than equals to $3$. 

For each singularity $s \in \mathcal{S}$, let $P_s$ denote its index. By the Euler-Poincar\'{e} formula (\cite{FLP}) and the fact that $P_s \geq 3$, we have 
$$ k_1 \geq \dfrac{n-k_2+4-4g}{2} > \dfrac{4\epsilon n - k_2 + 4}{2}.$$
The second inequality follows from our hypothesis $g < (1/4 - \epsilon) n$. 

Now we are going to divide the situation into two cases. First, let us assume that $k_2 < 2 \epsilon n$. Then $k_1 > \epsilon n + 2$. 

Suppose there are at least $N$ real branches attached at the monogons containing punctures. Then there are in total at least $k_1 N / 2$ many real branches in $\tau$. Since $k_1 > \epsilon n + 2$, we have 
$$ k_1 N / 2 \geq (N/2)(\epsilon n + 2), $$
which give us a lower bound on the number of real branches. 

On the other hand, by Fact (2) above, the number of real branches is bounded above by $9|\chi(S_{g,n})| = 9(3n/2 - 2)$. Hence, if $N$ is sufficiently large, we get a contradiction. 

This shows that the there exists a uniform number $N$ such that there exists at least one monogon where the number of attached real branches is bounded above by $N$. Since these real branches are permuted by $\phi$, this gives a uniform upper bound on $q$ in Fact (3). 

For the second case, let us assume $k_2 \geq 2 \epsilon n$. Again, suppose there are at least $N$ real branches at each of the bigons containing punctures. Then there are at least $Nk_2/2 \geq N\epsilon n$ many real branches in $\tau$. Since the number of real branches is bounded above by $9(3n/2 - 2)$, for sufficiently large $N$, we get a contradiction as before. This implies there exists a uniform number $N$ such that $\tau$ has at least one bigon containing a puncture where less than $N$ real branches are attached. This again gives a uniform upper bounded on $q$. 

In either case, since $q$ is uniformly bounded by a constant, we get the desired result. 

\end{proof}

The main result of this paper is that $L_C(\mathcal{I}(S_g^n, P)) \asymp 1/|\chi(S_g^n)|$ for any partition $P$ of $\partial S_g^n$. Hence, from the perspective of Conjecture \ref{conj:punctured} (1), $\Mod(S_g^n)$ and their proper normal subgroups seem to be not  distinguishable by the least asymptotic translation length on the curve complex unlike the case of closed surfaces. 

On the other hand, in the case of punctured surface, it is much more hopeful. If one can show Conjecture \ref{conj:punctured}, then it would be a strong partial evidence toward that $\Mod(S_{g,n})$ and their proper normal subgroups are distinguishable by the least asymptotic translation length on the curve complex. 

Furthermore, one way to define the Torelli subgroup $\mathcal{I}_{g,n}$ of $\Mod(S_{g,n})$ is to define it as $\Phi_\ast (\mathcal{I}(S_g^n, P))$ where $P$ is the maximal partition of the boundary components. Then $\mathcal{I}_{g,n}$. In this case, our main result implies that $L_C(\mathcal{I}_{g,n}) \asymp 1/|\chi(S_{g,n})|$, hence it is an additional partial evidence. 

In summary, we propose the following 
\begin{conj} 
There exists a uniform constant $C >0$ such that for any proper normal subgroup $H$ of $\Mod(S_{g,n})$, $L_C(H) \geq C/|\chi(S_{g,n})|$. 
\end{conj} 

This is an extended version of Question 1.2 in \cite{BaikShin18}.

\nocite{davis2000poincare}
\nocite{brown2012cohomology}

%
%

\medskip
\bibliographystyle{alpha} 
\bibliography{torelli}

\end{document}